\documentclass[11pt]{amsart}
\usepackage{a4, amsmath,  amsthm,  amsfonts, mathrsfs, latexsym, amssymb, 
pstricks, pst-grad, graphicx, tabularx, setspace, longtable}
\usepackage{enumerate}
\usepackage{hyperref}
\usepackage{color}
\usepackage{psfrag}
\usepackage{amssymb}
\scshape
\date{}
\usepackage{tikz}
\usetikzlibrary{arrows}
\usepackage{subcaption}
\usetikzlibrary{automata}
\newtheorem{thm}{\bf Theorem}[section]

\newtheorem{defn}[thm]{\bf Definition}

\newtheorem{rem}[thm]{\bf Remark}

\begin{document}
\title[Eulerian character degree graphs of solvable groups]{EULERIAN CHARACTER DEGREE GRAPHS OF SOLVABLE GROUPS}
 
\author [G. Sivanesan]{G. Sivanesan}
\address{Department of Mathematics, Government College of Engineering, Salem 636011, Tamil Nadu, India, ORCID:0000-0001-7153-960X.}
\email{sivanesan@gcesalem.edu.in}

\author [C. Selvaraj]{C. Selvaraj}
\address{Department of Mathematics, Periyar University, Salem 636011, Tamil Nadu, India, 0000-0002-4050-3177}
\email{selvavlr@yahoo.com}

\author[T. Tamizh Chelvam]{T. Tamizh Chelvam}
\address{Department of Mathematics, Manonmaniam Sundaranar University, 
Tirunelveli 627 012, Tamil Nadu, India, ORCID:0000-0002-1878-7847.}
\email{tamche59@gmail.com}

\keywords{character graph; Eulerian graph; regular graph; finite solvable group}
\subjclass[2000]{05C45, 20C15, 05C25}
 
\begin{abstract} Let $G$  be a finite group,  let $Irr(G)$ be the set of all complex irreducible characters of $G$ and let $cd(G)$ be the set of all degrees of characters in $Irr(G).$  Let $\rho(G)$ be the set of primes that divide degrees in $cd(G).$ The character degree graph $\Delta(G)$ of $G$ is the simple undirected graph with vertex set $\rho(G)$ and in which two distinct vertices $p$ and $q$ are adjacent  if there exists a  character degree $r \in cd(G)$ such that $r$ is divisible by the product $pq.$ In this paper, we obtain a necessary condition for the character degree graph $\Delta(G)$ of a finite solvable group $G$ to be Eulerian. Morresi Zuccari [C. P. Morresi  Zuccari, Regular character degree graphs, \textit{J. Algebra} {\bf 411}  (2014), 215--224.] proved that if $\Delta(G)$ is a non-complete and regular character degree graph, then $\Delta(G)$ is $n-2$ regular. Now, we prove the converse, that is, every $n-2$ regular graph with $n\geq 4$ and $n$ is even, is a character degree graph of some finite solvable group $G$. 
\end{abstract}
\date{}
\maketitle
 
\section{Introduction}  Throughout this paper, $G$ will be a finite solvable group with identity $1.$ We denote the set of complex irreducible characters of $G$ by $Irr(G).$ Here $cd(G) =\{\chi(1) \mid \chi\in Irr(G)\}$ is the set of all distinct degrees of irreducible characters in $Irr(G).$ Let $\rho(G)$ be the set of all primes that divide degrees in $cd(G).$  There is a vast literature  devoted to study of ways through which one can associate a graph with a group and that literature can be used  for investigating the algebraic structure of groups  using graph theoretical properties of  associated graphs. One of these graphs is the character degree graph $\Delta(G)$ of $G.$  In fact, $\Delta(G)$ is an undirected simple graph with vertex set $\rho(G)$ in which $p,q\in\rho(G)$ are joined by an edge if there exists a character degree $\chi(1)\in cd(G)$ which is divisible by $pq.$  This graph was first defined in  \cite{Manz2} and studied by many authors (see \cite{Lewis5}). When $G$ is a solvable group, some interesting results on the character graph of $G$ have been obtained. Actually, Manz\cite{Manz1} proved that $\Delta(G)$ has at most two connected components. Also, Manz et al.\cite{Manz3} have proved that diameter of $\Delta(G)$ is at most $3$. Mahdi Ebrahimi et al.\cite{EI} proved that the character graph $\Delta(G)$ of a solvable group $G$ is Hamiltonian if and only if $\Delta(G)$ is a block with at least $3$ vertices. Motivated by these studies on character degree graphs of solvable groups, we obtain a  necessary condition for the character degree graph $\Delta(G)$ of a solvable group $G$ to be Eulerian. In fact, we obtain characterizations for the character degree graph to be Eulerian through possibilities for the diameter of $\Delta(G).$

\section{Preliminaries}
In this section, we present some preliminary results which are used in the paper. All graphs are assumed to be simple, undirected and finite. Let $\Gamma$ be a graph with vertex set $V(\Gamma)$ and edge set $E(\Gamma)$. If $\Gamma$ is connected, then the distance $d(u,v)$ between two distinct vertices $u,v\in V(\Gamma)$ is the length of the shortest path between them. The supremum of all distance between possible pairs of distinct vertices is known as the diameter of the graph. A complete graph with $n$ vertices in which any of the two distinct vertices are adjacent is denoted by $ K_n.$  A cycle with $n$ vertices is denoted by $C_n.$ The vertex connectivity $k(\Gamma)$ of $\Gamma$ is defined to be the minimum number of vertices whose removal from $\Gamma$ results in a disconnected subgraph of $\Gamma.$ A cut vertex $v$ of a graph $\Gamma$ is a vertex  such that the number of connected components of $\Gamma -v$ is more than the number of connected components of $\Gamma$. Similar definition is applicable in the case of cut edge of graph $\Gamma.$  A maximal connected subgraph without a cut vertex is called a block. By their maximality, different blocks of $\Gamma$ overlap in at most one vertex, which is then a cut vertex. Thus, every edge of $\Gamma$ lies in a unique block and $\Gamma$ is the union of its blocks. The degree of a vertex $v$ in $\Gamma$ is the number of edges incident with $v$ and the same is denoted by $ d(v)$ or deg$v$. A graph $\Gamma$ is called $k$-regular, if the degree of each vertex is $k$. 
A graph $\Gamma$ is said to be Eulerian if it contains a cycle containing all vertices of $\Gamma.$ The famous characterization~\cite[Theorem 3.5]{BM} states that a connected graph is an Eulerian graph if and only if all the vertices are of even degree.

We will take into account the following well known facts concerning character degree graphs and they are needed in the next section.

\begin{rem}\normalfont\label{rem2.1} For results regarding $\Delta(G),$ we start with P\'{a}lfy's three prime theorem on the character degree graph of solvable groups. P\'{a}lfy theorem~\cite[Theorem, pp. 62]{PP} states that given a solvable group $G$ and any three distinct vertices of $\Delta(G)$, there exists an edge incident with the other two vertices. On applying P\'{a}lfy's theorem, $\Delta(G)$  has at most two connected components. 
\end{rem}
\begin{rem}\normalfont\label{rem2.2} Let $G$ be a solvable group. Then $diam(\Delta(G))\leq 3$\cite[Theorem 3.2]{Manz3}. Assume that the diameter $diam(\Delta(G))=3$ and let $r,s\in\rho(G)$ be two distinct vertices in $\Delta(G)$ such that distance $d(r,s)=3.$  When $\Delta(G)$ has exactly diameter $3$ and contains at least $5$ vertices,  Lewis~\cite[pp. 5487]{Lewis2} proved that the vertex set $\rho(G)$ of $\Delta(G)$ can be partitioned into $\rho_1,\rho_2,\rho_3$ and $\rho_4.$ Actually $\rho_4$ is the set of all vertices of $\Delta(G)$ which are at distance $3$ from the vertex $r$ and so we have that $s\in \rho_4,$ $\rho_3$ is  the set of all vertices of $\Delta(G)$ which are distance $2$ from the vertex $r,$ $\rho_2$ is the set of all vertices that are adjacent to vertex $r$ and adjacent to some prime in $\rho_3$ and $\rho_1$ consists of $r,$ the set of all vertices which are adjacent to $r$ and not adjacent to any vertex in $\rho_3.$  This implies that, no vertex in $\rho_1$ is adjacent to any vertex in $\rho_3\cup\rho_4$ and no vertex in $\rho_4$ is adjacent to any vertex in $\rho_1\cup\rho_2,$ every vertex in $\rho_2$ is adjacent to some vertex in $\rho_3$ and vice versa, and $\rho_1\cup\rho_2$ and $\rho_3\cup\rho_4$ both induce complete subgraphs of $\Delta(G).$
\end{rem}
\begin{rem}\normalfont\label{rem2.3} 
Huppert\cite[pp. 25]{HU} listed all possible graphs $\Delta(G)$ for solvable groups $G$ with at most $4$ vertices.   In fact, every graph with $3$ or few vertices that satisfies P\'{a}lfy's condition occurs as $\Delta(G)$ for some solvable group $G$\cite[pp. 184]{Lewis4}. 
\end{rem}
\begin{thm}\cite[Lemma 2.7]{EI}\label{2.2} Let $G$ be a group with $|\rho(G)|\geq 3.$  If $\Delta(G)$ is not a block and the diameter of $\Delta(G)$ is at most $2$, then each block of $\Delta(G)$ is a complete graph.
\end{thm}

\begin{thm} \cite[Theorem 5]{Zha}\label{2.3}  The graph with four vertices in Figure 1 is not the character degree graph of a solvable group.
	
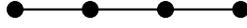
\begin{figure}[ht]
\centering
\begin{tikzpicture}
\draw[fill=black] (0,0) circle (3pt);
 \draw[fill=black] (1,0) circle (3pt);
 \draw[fill=black] (2,0) circle (3pt);
 \draw[fill=black] (3,0) circle (3pt);
\draw[thick] (0,0) --(1,0) --(2,0) --(3,0);
\end{tikzpicture}
 \caption{Graph with four vertices}
\end{figure}
\end{thm}

\begin{thm}\cite[Theorem 1.1]{Lewis6}\label{2.4}  Let $G$ be a solvable group. Then $\Delta(G)$ has at most one cut vertex.
\end{thm}	

\begin{thm}\cite[Theorem A]{MZ}\label{2.4.1} If $\Delta(G)$ is a non-complete and regular character degree graph of a finite solvable group $G$ with $n$ vertices, then $\Gamma(G)$ is $n-2$ is a regular graph.
\end{thm}

\begin{thm}\cite[Lemma 2.1]{Lewis6}\label{2.5}  Let $G$ be a solvable group and assume that $\Delta(G)$ has diameter $3$. Then $G$ is $1$-connected if and only if $|\rho_2|=1$ in the diameter $3$ partition of  $\rho(G)$. In this case, if $p$ is the unique prime in $\rho_2$, then $p$ is also the unique cut vertex for $\Delta(G)$. In particular, $\Delta(G)$ has at most one cut vertex.
\end{thm}

\begin{thm}[~\cite{Lewis4} Lemma $2.1$]\label{2.6}  Let $p, q_1, . . . , q_n$ be distinct primes so that $p$ is odd. Then there is a
solvable group $G$ of Fitting height $2$ so that $\Delta(G)$ has two connected components: $\{p\}$ and $\{q_1, . . . , q_n\}$.
\end{thm}

\begin{rem}{\bf Operation D:}\label{rem2.6} \normalfont Consider a character degree graph $\Delta(G).$ Select two distinct primes $p$ and $q$ such that $p$ is odd prime and $p,q\not\in \rho(G).$ By Theorem~\ref{2.6}, there exists a solvable group $H$ of Fitting height $2$ and $\Delta(H)$ is the graph containing two isolated vertices. Now one can make a obtain the direct product $\Delta(G\times H).$  Given $\Delta(G),$ the construction of  such a direct product $\Delta(G\times H)$ is used in the proof of Theorem~\ref{4.1} repeatedly. We denote this process as Operation D.
\end{rem}
 
\begin{thm}\cite[Lemma 4.1]{EII}\label{2.7}
Let $G$ be a solvable group  where the diameter of  $\Delta(G)$ is $3$ with Lewis partition $\rho(G)=\rho_1\cup\rho_2\cup\rho_3\cup\rho_4.$ Then $\Delta(G)$ is a block if and only if $|\rho_2|,|\rho_3|\geq 2.$
\end{thm}

\begin{defn}\cite[pp. 502]{BL}\label{def2.8} Using direct product, one can find bigger groups from smaller groups. The same may be used to construct higher order character degree graphs. For two groups $A$ and $B$ where $\rho(A)$ and  $\rho(B)$ are disjoint,  we have that $\rho(A\times B) =  \rho(A) \cup \rho(B)$.  Define an edge between vertices $p$ and $q$ in $\rho(A\times B)$  if any of the following is satisfied:
\begin{itemize}
\item [\rm (i)] $p , q\in \rho (A)$ and there is an edge between $p$ and $q$ in $\Delta(A);$
\item [\rm (ii)] $p, q \in \rho (B)$ and there is an edge between $p$ and $q$ in $\Delta(B);$
\item [\rm (iii)] $p\in\rho (A)$ and $q\in\rho (B);$
\item [\rm (iv)]  $p \in\rho (B) $ and $q\in\rho (A).$ 
\end{itemize}
Now we get a higher order character degree graph and it is called direct product.
\end{defn}

\section{Eulerian Character Degree Graphs}
 We now prove below  the character degree graph $\Delta(G)$ to be Eulerian when the diameter of $\Delta(G)$ is at most 3.
\begin{thm}\label{3.1}
Let $G$ be a solvable group and let $\Delta(G)$  be the character degree graph of  $G$  with $n$ vertices. If any of the following conditions is true, then  $\Delta(G)$ is an Eulerian  graph.
\begin{itemize}
\item [\rm (i)]  $\Delta(G)$ is a complete graph with  $n \geq 3$ is odd;
\item[\rm (ii)]  $\Delta(G)$ is a non-complete regular graph with $n \geq 4$ and $n$ is even;
\item[\rm (iii)] $\rho(G)\geq 3,$ $\Delta(G)$ is not a block, each block of $\Delta(G)$ contains odd number of vertices and diameter of $\Delta(G)$ is at most 2.
\end{itemize}
\end{thm}
\begin{proof} (i) By assumption, $\Delta(G)$ is a complete graph with odd number of vertices. Hence each vertex in $\Delta(G)$ is of even degree and so $\Delta(G)$ is Eulerian.  

(ii) By assumption, $\Delta(G)$ is a non-complete regular graph. By Theorem~\ref{2.4.1} $\Delta(G)$ is a $n-2$ regular graph. This gives that each vertex in $\Delta(G)$ is of even degree and so $\Delta(G)$ is Eulerian.

(iii) Suppose $\Delta(G)$ has diameter at most $2.$ By Theorem~\ref{2.2}, each block of $\Delta(G)$ is a complete graph. By the assumption that each block contains odd number of vertices. So, all the vertices of $\Delta(G)$ are even degree. Therefore $\Delta(G)$ is an Eulerian graph. \hfill $\square$ 
\end{proof}
\begin{rem}\label{rm3.3}\normalfont The assumption that the diameter of $\Delta(G)$ is at most 2  is  essential in Theorem~\ref{3.1}(iii). For, assume that $\Delta(G)$ has diameter $3$. By assumption $\Delta(G)$ is not a block, and each block of $\Delta(G)$ contain odd number of vertices. $\Delta(G)$ has two blocks say $B_1$ and $B_2$ and it has a cut vertex. This cut vertex is belongs to $\rho_2$. By  Theorem~\ref{2.5}, $|\rho_2| =1$. Since $\rho_1 \cup \rho_2$ and $\rho_3 \cup \rho_4$ induce complete subgraphs, each vertex  in  $\rho_3 \cup \rho_4$ is of even degree.  Since $|\rho_2| = 1$, and every vertex in  $\rho_3$ is adjacent to some vertex in $\rho_2$, each vertex  in $\rho_3$ is of odd degree.  Hence  $\Delta(G)$ is not an Eulerian.
\end{rem}
\begin{rem}\normalfont \label{rm3.4} The condition that each block should contain odd number of vertices in  Theorem~\ref{3.1}(iii) is necessary. For instance, the graph in Figure 2 is the character degree graph of a solvable group with diameter $2$ \cite[pp. 503]{BL}. This graph contains two blocks and one block contains odd degree vertices. Hence it is not an Eulerian graph.

\begin{figure}[ht]
\centering
\begin{tikzpicture}
\draw[fill=black] (0,0) circle (3pt);
\draw[fill=black] (1,1) circle (3pt);
 \draw[fill=black] (0,2) circle (3pt);
 \draw[fill=black] (4,0) circle (3pt);
\draw[fill=black] (2,1) circle (3pt);
 \draw[fill=black] (4,2) circle (3pt);
\draw[thick] (0,0) --(0,2) --(2,1) --(4,2);
\draw[thick] (0,0) --(2,1) --(4,0) --(4,2);
\draw[thick] (0,0) --(1,1) --(2,1) ;
\draw[thick] (0,0) --(1,1) --(0,2) ;
\end{tikzpicture}
\caption{Graph with six vertices}
\end{figure}
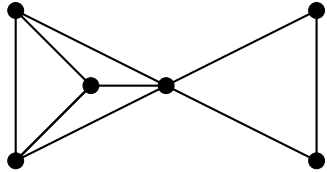
 
\end{rem}
We now prove below that the character degree graph $\Delta(G)$ is Eulerian when the diameter of $\Delta(G)$ is $3.$
 
\begin{thm}\label{3.2} Let $G$ be a solvable group with diameter of $\Delta(G)$ is $3.$ Let the  Lewis' partition of $G$ be $\rho(G)=\rho_1\cup\rho_2\cup\rho_3\cup\rho_4.$ Then  $\Delta(G)$ is an Eulerian graph if and only if the following conditions are hold:
\begin{itemize}
\item [\rm (i)] $\Delta(G)$ is a block;
\item [\rm (ii)] Both $|\rho_1 \cup \rho_2|$ and $|\rho_3 \cup \rho_4|$ are odd;
\item [\rm (iii)] The subgraph of $\Delta(G)$ induced by $\rho_2 \cup \rho_3$ is Eulerian.
\end{itemize}
\end{thm}
\begin{proof}
Assume that $\Delta(G)$ be Eulerian graph.

(i) If $\Delta(G)$ is not a block, by Remark~\ref{rm3.3} is not Eulerian. Hence $\Delta(G)$ is a block.
 
(ii) We claim that both $|\rho_1 \cup \rho_2|$ and $|\rho_3 \cup \rho_4|$ are odd. Suppose not, let us assume that $|\rho_1\cup \rho_2|$ is even. As mentioned in Remark~\ref{rem2.2}, $\rho_1\cup \rho_2$ is a complete graph and no prime in $\rho_1$ is adjacent to any prime in $\rho_3\cup \rho_4.$ Hence the vertices in $\rho_1$ are of odd degree and so $\Delta(G)$is not Eulerian, which is a contradiction. Therefore $|\rho_1 \cup \rho_2|$ is odd. 

Suppose that $|\rho_3\cup \rho_4|$ is even. As mentioned in Remark~\ref{rem2.2}, $\rho_3\cup \rho_4$ is a complete graph and no prime in $\rho_4$ is adjacent to any prime in
$\rho_1\cup \rho_2.$ Hence the vertices in $\rho_4$ are of odd degree and so $\Delta(G)$is not Eulerian, which is a contradiction. Therefore $|\rho_3 \cup \rho_4|$ is odd. 

(iii) Suppose the subgraph induced by $\rho_2\cup \rho_3$ is non Eulerian.  One can refer Remark~\ref{rem2.2} for the following facts. No prime in $\rho_1$ is adjacent to any prime in $\rho_3\cup \rho_4.$ Similarly no prime in $\rho_4$ is adjacent to any prime in $\rho_1\cup \rho_2.$  Since $\rho_1 \cup \rho_2$ and $\rho_3 \cup \rho_4$ are complete graphs with odd number of vertices. Hence each vertex in $\rho_2$ and $\rho_3$ is of even degree in sub graphs induced by $\rho_1 \cup \rho_2$ and $\rho_3 \cup \rho_4$. Therefore  the vertices in $\rho_1$ and $\rho_4$ are of even degree. But the degree of each vertex in $\rho_2$ is the sum of the degree of that vertex in $\rho_2$ in the subgraph induced by $\rho_2\cup \rho_3$ and degree of that vertex in $\rho_2$ in  the subgraph induced by $\rho_1\cup \rho_2$.  Similarly degree of each vertex in $\rho_3$ is the sum of degree of that vertex in $\rho_3$ in the subgraph induced by $\rho_2\cup \rho_3$ and degree of that vertex in $\rho_3$ in the subgraph induced by $\rho_3\cup \rho_4$. By assumption $\rho_2\cup \rho_3$ is non Eulerian and hence there are vertices of odd degree in $\rho_2$ or $\rho_3$ in the subgraph induced by $\rho_2\cup \rho_3.$ Hence some vertices in $\rho_2$  or  $\rho_3$ are of odd degree, which is a contradiction to $\Delta(G)$ is an Eulerian graph. Therefore the vertex induced subgraph $\rho_2\cup \rho_3$ is Eulerian.

Conversely, assume that conditions (i)-(iii) are true. According to Lewis partition, the subgraphs induced by $\rho_1\cup \rho_2$ and $\rho_3\cup \rho_4$ are complete subgraphs of $\Delta(G),$  no prime in $\rho_1$ is adjacent to any prime in $\rho_3\cup \rho_4$ and  no prime in $\rho_4$ is adjacent to any prime in $\rho_1\cup \rho_2.$ Since both $|\rho_1 \cup \rho_2|$ and $|\rho_3 \cup \rho_4|$ are odd, vertices in $\rho_1$ and $\rho_4$ are of even degree. Since $\Delta(G)$ is a block, by Theorem~\ref{2.7} we get that $|\rho_2|, |\rho_3|\geq 2.$ Since the subgraph induced by $\rho_2 \cup \rho_3$ is Eulerian, we get that vertices in $\rho_2$ and $\rho_3$ are of even degree. According to Lewis partition, all subsets $\rho_1, \rho_2, \rho_3$ and $\rho_4$ in $\rho(G)$  are non-empty disjoint subsets. Therefore all the vertices in $\rho(G)$ are even degree. Hence $\Delta(G)$ is Eulerian graph. \hfill $\square$ 
\end{proof}

As stated in Theorem~\ref{2.4}, Lewis has proved that $\Delta(G)$ has at most one cut vertex.  We shall prove that $\Delta(G)$ has at most two cut edges in the following theorem.

\begin{thm}\label{3.3} Let $G$ be a solvable group and let the character degree graph $\Delta(G)$ be connected. Then $\Delta(G)$ has at most two cut edges.
\end{thm}
\begin{proof}  \textbf{Case 1.} If $\Delta(G)$ is a block, then there is no cut edge. 

\textbf{Case 2.} If $\Delta(G)$ is not a block and  $\Delta(G)$ has two blocks say  $B_1$ and $ B_2$ with both blocks are not $K_2,$ then there is no cut edge.

\textbf{Case 3.} If $\Delta(G)$ is not a block and $\Delta(G)$ has two blocks say  $B_1$ and $ B_2$ with at least one of them say  $B_1$  is $K_2.$  Then $\Delta(G)$ has one cut edge.

\textbf{Case 4.} If $\Delta(G)$ is not a block and $\Delta(G)$ has two blocks say $B_1$ and $ B_2$ with both them as $K_2.$ Then $\Delta(G)$ has two cut edges.
 \hfill $\square$ 
\end{proof}

Having attempted Eulerian characterization for character degree graph, we prove below that the converse of a result by Morresi Zuccari\cite{MZ}. Actually Morresi Zuccari proved that if $\Gamma(G)$ is a non-complete and regular character degree graph of a finite solvable group $G,$ then $\Gamma(G)$ is $n-2$ regular. We prove the converse of this through a classification of  character degree graphs with six vertices obtained by Bisler\cite{BL}. In fact Bisler has effectively used the direct product of two character degree graphs. We also utilize the same idea of direct product in order to prove that every $n-2$ regular graph is the character degree graph of some solvable group.
 
\begin {thm}\label{3.5} Let $n\geq 4$ be an even integer. Every $n-2$ regular graph is the character degree graph $\Delta(G)$ for some solvable group $G.$
\end{thm}
\begin{proof} Let us prove the result using the method of induction on the number of vertices of $\Delta(G).$ 

\textbf{Case 1.} Let $n=4.$  By assumption $\Delta(G)$ is the  square graph on four vertices and it is the character degree graph as shown in \cite[pp.184]{Lewis5}.

\textbf{Case 2.} Let $n=6.$ Now $\Delta(G)$ is a $4$-regular graph with six vertices and is the character degree graph as shown in \cite[pp. 502-503]{BL}.

\begin{figure}[ht]
    \centering 
     \begin{subfigure}{.5\textwidth} 
         \centering 
         \begin{tikzpicture} 
           \draw[fill=black] (2,3) circle (3pt); 
\draw[fill=black] (0.5,2) circle (3pt);
 \draw[fill=black] (0.5,0) circle (3pt);
 \draw[fill=black] (2,-1) circle (3pt);
 \draw[fill=black] (3.5,0) circle (3pt);
 \draw[fill=black] (3.5,2) circle (3pt);
\draw[thick] (2,3) --(0.5,2)--(0.5,0)--(2,-1)--(3.5,0)--(3.5,2)--(2,3);
\draw[thick] (2,3) --(0.5,0)--(3.5,0)--(2,3);
\draw[thick] (0.5,2) --(2,-1)--(3.5,2)--(0.5,2);   
         \end{tikzpicture} 
         \caption{$\Delta(H)$} 
     \end{subfigure}
     \begin{subfigure}{.5\textwidth} 
         \centering 
         \begin{tikzpicture} 
             \draw[fill=black] (3.5,0) circle (3pt); 
\draw[fill=black] (5.5,0) circle (3pt);
         \end{tikzpicture} 
         \caption{$\Delta(K)$} 
     \end{subfigure} 
\begin{subfigure}{.5\textwidth} 
         \centering 
         \begin{tikzpicture} 
\draw[fill=black] (-0.5,0) circle (3pt);
 \draw[fill=black] (4.5,0) circle (3pt);
 \draw[fill=black] (0.5,1.5) circle (3pt);
 \draw[fill=black] (3.5,1.5) circle (3pt);
 \draw[fill=black] (2,2.5) circle (3pt);
 \draw[fill=black] (0.5,-1.5) circle (3pt);
 \draw[fill=black] (2,-2.5) circle (3pt);
\draw[fill=black] (3.5,-1.5) circle (3pt);
\draw[thick] (2,2.5) --(0.5,1.5)--(-0.5,0)--(0.5,-1.5)--(2,-2.5)--(3.5,-1.5)--(4.5,0)--(3.5,1.5)--(2,2.5);
\draw[thick] (0.5,1.5) --(3.5,1.5)--(2,-2.5)--(0.5,1.5);
\draw[thick] (0.5,-1.5) --(2,2.5)--(3.5,-1.5)--(0.5,-1.5);   
\draw[thick] (0.5,1.5) --(0.5,-1.5)--(3.5,-1.5)--(3.5,1.5);
\draw[thick] (-0.5,0) --(3.5,1.5)--(3.5,-1.5)--(-0.5,0); 
\draw[thick] (-0.5,0) --(2,2.5)--(4.5,0)--(2,-2.5)--(-0.5,0);
\draw[thick] (4.5,0) --(0.5,1.5)--(0.5,-1.5)--(4.5,0);
         \end{tikzpicture} 
         \caption{ $\Delta($H$ \times $K$)$ } 
     \end{subfigure}
\caption{ Direct product of character degree graphs } 
\end{figure}
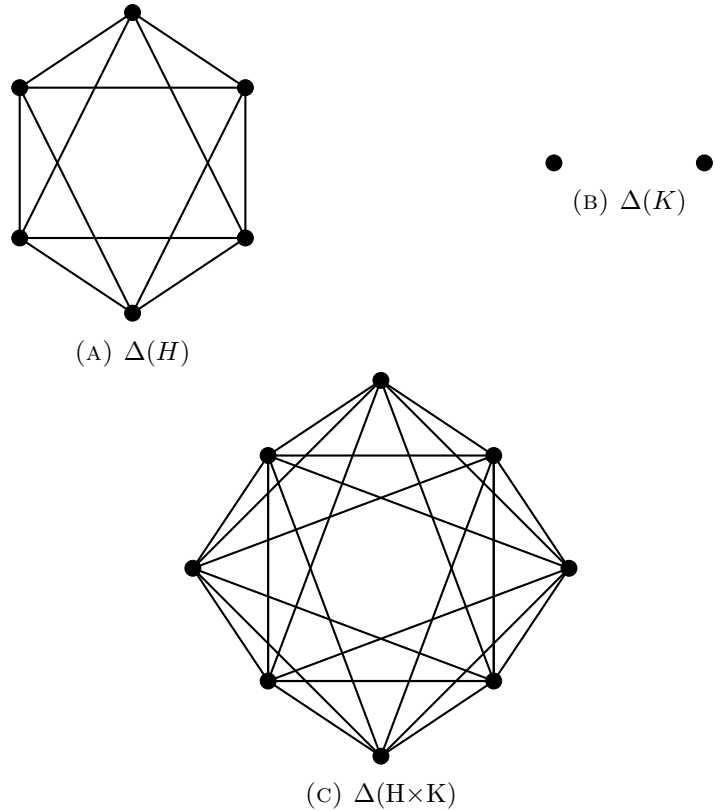 

\textbf{Case 3.} Let $n=8.$  Using direct product construction,  we can produce a $6$-regular graph with $8$ vertices.  For, consider $\Delta(H)$ as given in Figure 3(A) and it is a $4$-regular graph with six vertices.  Choose two distinct primes $p, q$ so that $p$ is odd and disjoint from $\rho(H)$.  By Theorem~\ref{2.6},  then there is a solvable group $K$ of Fitting height $2$ so that $\Delta(K)$ has two connected components: $\{p\}$ and $\{q\}$.  Now  $\Delta(K)$ is  as  in Figure 3(B).  Consider the direct product $\Delta(H\times K)$ as given in Definition~\ref{def2.8}. Then the direct product is nothing but the graph given in Figure 3(C). Hence $\Delta(H \times K)$ is a 6-regular graph with 8 vertices as required.

\textbf{Case 4.} Let $n=10.$ Let $\Delta(H)$ be a 6-regular graph with 8 vertices and let $\Delta(K)$ be the graph with two isolated vertices which are disjoint from $\rho(H)$. Now using direct product as in Case 3, $\Delta(H\times K)$ is a 8-regular graph with 10 vertices.
 
Assume that as induction hypothesis, there exists a $n-4$ regular graph with $n-2$ vertices. Let $\Delta(H)$ be this $n-4$ regular graph with $n-2$ vertices and $\Delta(K)$ be the graph with two isolated vertices which are disjoint from $\rho(H)$. Then $\Delta(H \times K)$ is a $n-2$ regular graph with $n$ vertices for the solvable group $H\times K.$ \hfill $\square$ 
\end{proof}

\section{Number of Eulerian Character Degree Graphs}
In this section, we obtain the number of Eulerian character degree graphs in terms of number of vertices. Actually, we obtain below that the number of character degree graphs with $n$ vertices ($n\geq 6$  and $n$ is even) which are non regular Eulerian   by assuming that the diameter of $\Delta(G)$ is two. Let $\lfloor x \rfloor$ be the greatest integer not exceeding the real number $x.$
 
\begin{thm} \label{4.1} Let $G$ be a finite solvable group.  Let $\mathcal{G}_n$ be the class of all $\Delta(G)$ where $\Delta(G)$ is a non-regular block with diameter $2$ containing $|\rho(G)|=n\geq 6$ vertices and $n$ is even. The number of Eulerian character degree graphs in $\mathcal{G}_n$ is at least $\lfloor \frac{n-4}{2} \rfloor$ $+$ $(\lfloor  \frac{n}{6} \rfloor-1).$ 
\end{thm}

\begin{proof}  We prove the result by identifying required number of Eulerian character degree graphs in $\mathcal{G}_n.$  

\textbf{Step 1.} First let us find Eulerian character degree graphs in the class $\mathcal{G}_n$ where $6\leq n \leq 12$ and $n$ is even. In this proof all the  character degree graphs are constructed via direct product via repeated application Operation D\ref{rem2.6} .

\textbf{Case 1.1} $|\rho(G)| = 6.$ M. Bissler~\cite[pp. 503] {BL} classified all character degree graphs with $6$ vertices expect $9$ graphs  listed in \cite[pp. 509]{BL}. But none of these $9$ graphs are Eulerian.  On the other hand, among  twelve graphs, there exists only one non regular graph with all the vertices as even and the same is given in Figure 4.  Thus there exists only one graph in $\mathcal{G}_6$ and hence the number of graphs in  $\mathcal{G}_6$ is at least one and is equal to $\lfloor \frac{6-4}{2} \rfloor+(\lfloor  \frac{6}{6} \rfloor-1).$ 
 
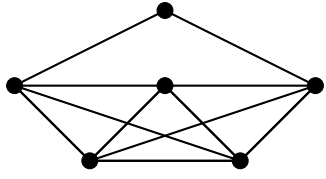
\begin{figure}[ht]
\centering
\begin{tikzpicture}
\draw[fill=black] (-1,0) circle (3pt);
\draw[fill=black] (1,0) circle (3pt);
 \draw[fill=black] (0,1) circle (3pt);
\draw[fill=black] (-2,1) circle (3pt);
\draw[fill=black] (2,1) circle (3pt);
 \draw[fill=black] (0,2) circle (3pt);
\draw[thick] (0,2) --(-2,1)--(-1,0) --(1,0)--(2,1) --(0,2);
\draw[thick] (-2,1) --(-1,0) --(0,1) --(-2,1);
\draw[thick] (0,1) --(1,0) --(2,1) --(0,1);
\draw[thick] (-1,0) --(1,0) --(0,1) --(-1,0);
\draw[thick] (-1,0) --(2,1) ;
\draw[thick] (-2,1) --(1,0) ;
\end{tikzpicture}
 \caption{Graph with six vertices}
\end{figure}

\textbf{Case 1.2} $|\rho(G)| = 8.$ Let $\Delta(G)$ be the Eulerian character degree graph given in Figure 4. After applying Operation D~\ref{rem2.6}, we get a direct product $\Delta(G\times H)$  with $8$ vertices in which $1$ vertex is of degree $4$ and the remaining $7$  vertices are of degree $6$. Hence this direct product gives a graph in $\mathcal{G}_8.$

Again by Theorem~\ref{2.6}, one can have a $\Delta(G)$ with has two complete connected components: one having an isolated vertex and the other is a complete graph on $5$ vertices corresponding to $6$ different primes.  Now after applying  Operation D~\ref{rem2.6}, we get a direct product $\Delta(G\times H)$  with $8$ vertices in which $1$ vertex is of degree $2$ and the remaining $7$  vertices are of degree $6$. Note that this gives another graph in $\mathcal{G}_8.$ 

Thus the  number of graphs in $\mathcal{G}_8$ at least $2$ and it is equal to $\lfloor \frac{8-4}{2} \rfloor+(\lfloor  \frac{8}{6} \rfloor-1).$  

\textbf{Case 1.3} $|\rho(G)| = 10.$ There are two Eulerian character degree graphs in $\mathcal{G}_8$ with $8$ vertices as  constructed in Case $1.2$.  Let us take each of these graphs as $\Delta(G)$.  After applying Operation D~\ref{rem2.6}, we get two direct products $\Delta(G\times H)$  with $10$ vertices. In the first graph, one vertex is of degree $4$ and the remaining $9$ vertices are of  degree $8$.  In the second graph, one vertex is of degree $6$ and the remaining $9$ vertices are of degree $8$. 
Thus we get two graphs in $\mathcal{G}_{10}.$
 
Again by Theorem~\ref{2.6}, one can have a  $\Delta(G)$ which has two complete connected components: one having an isolated vertex and the other is a complete graph on $7$ vertices corresponding to $8$ different primes.  Now after applying  Operation D~\ref{rem2.6}, we get a direct product $\Delta(G\times H)$  with $10$ vertices in which $1$ vertex is of degree $2$ and the remaining $9$  vertices are of degree $8$. Note that this gives another graph in $\mathcal{G}_{10}.$ 

Thus the  number of graphs in $\mathcal{G}_{10}$ is at least $3$ and it is equal to $\lfloor \frac{10-4}{2} \rfloor+(\lfloor  \frac{10}{6} \rfloor-1).$

\textbf{Case 1.4} $|\rho(G)| = 12.$ There are three Eulerian character degree graphs  in $\mathcal{G}_{10}$ with $10$  as constructed in Case 1.3.  Let us take each of these graphs as  $\Delta(G)$. After applying Operation D~\ref{rem2.6}, we get three direct products $\Delta(G\times H)$  with $12$ vertices.  In the first graph, one vertex is of  degree $4$ and the remaining $11$  vertices are of degree $10$.  In the second graph, one vertex is of  degree $6$ and the remaining $11$ vertices are of  degree $10$. In the third graph, one vertex is of degree $8$ and $11$ vertices are of degree $10$.  Thus we get three graphs in $\mathcal{G}_{12}.$
 
Again consider the six vertex graph given in Figure 4. Let us take this as $\Delta(G).$  It is the direct product of two graphs given in Figure 5.
 
\begin{figure}[ht]
    \centering 
     \begin{subfigure}{.5\textwidth} 
         \centering 
         \begin{tikzpicture} 
           \draw[fill=black] (0,4) circle (3pt); 
\draw[fill=black] (-2,0) circle (3pt);
 \draw[fill=black] (2,0) circle (3pt);
 \draw[fill=black] (0,2) circle (3pt);
\draw[thick] (-2,0) --(2,0)--(0,2)--(-2,0);  
         \end{tikzpicture} 
         \caption{$\Delta(G_1)$} 
     \end{subfigure}
     \begin{subfigure}{.5\textwidth} 
         \centering 
         \begin{tikzpicture} 
             \draw[fill=black] (3.5,0) circle (3pt); 
\draw[fill=black] (5.5,0) circle (3pt);
         \end{tikzpicture} 
         \caption{$\Delta(G_2)$} 
\end{subfigure} 
\caption{$\Delta(G_1\times G_2)$}
\end{figure}
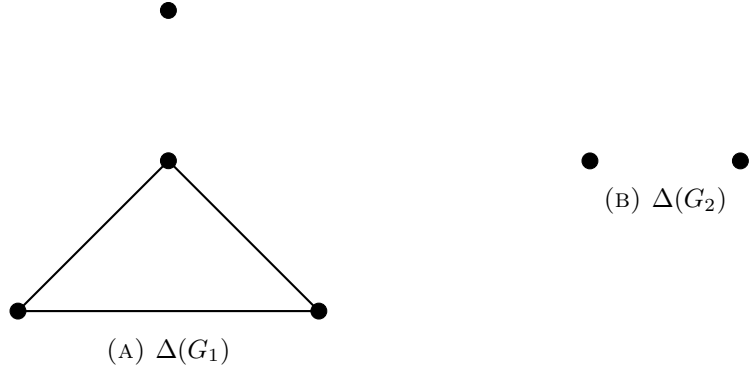 
 
Choose $6$ distinct prime numbers which are disjoint from $\rho(G)$. These primes shall give a graph $\Delta(H)$ as in Figure 4. The direct product of these graphs shall give $\Delta(G \times H)$ with $12$ vertices in which $2$ vertices are of  degree $8$ and the remaining $10$ vertices are of  degree $10.$ This is our fourth graph in  $\mathcal{G}_{12}.$

Now by Theorem~\ref{2.6}, one can find  a character degree graph  with two complete connected components: one having an isolated vertex and the other having $9$ vertices.  Taking this as $\Delta(G)$ and applying Operation~\ref{rem2.6}, we get a direct product $\Delta(G\times H)$ with $12$ vertices in which $1$ vertex is of  degree $2$ and $10$  vertices are of degree $10$. Thus we get one more graph in $\mathcal{G}_{12}.$

Thus the  number of graphs in $\mathcal{G}_{12}$ is at least $5$ and it is equal to $\lfloor \frac{12-4}{2} \rfloor+(\lfloor  \frac{12}{6} \rfloor-1).$  

\textbf{Step 2.} Let $n=6r+k$ where $k= 0, 2, 4.$ $r\geq 2$  and $n$ is even. Assume that as induction hypothesis the result is true for $n=6(r-1)+k$  where $k= 0, 2, 4.$ We shall prove the result through induction by considering the following possibilities and they cover all $n\geq 14$ and $n$ is even.

\textbf{Case 2.1} $r \geq 3$  and $k=2.$

By induction hypothesis, the result is true for the class $\mathcal{G}_{6(r-1)}$ and we have at least  $\lfloor \frac{6(r-1)-4}{2} \rfloor+(\lfloor  \frac{6(r-1)}{6} \rfloor-1)=4r-7$ each with  $6(r-1)$ vertices. Let us take each of these graphs as  $\Delta(G)$. After applying Operation D~\ref{rem2.6}, on each of these, we get  $4r-7$ direct products  $\Delta(G \times H)$ with $6(r-1)+2$ vertices in $\mathcal{G}_{6(r-1)+2}.$ 

Now by Theorem~\ref{2.6}, one can find  a character degree graph with two complete connected components: one having an isolated vertex and the other having $6r-7$ vertices.  Applying Operation D~\ref{rem2.6} on this, we get a  direct product  $\Delta(G \times H)$ with $6(r-1)+2$ vertices in which $1$ is of degree $2$ and  remaining $6(r-1)+1$  vertices are of degree $6(r-1)$.

Hence  the number of graphs in  $\mathcal{G}_{6(r-1)+2}$  is at least $4r-6$ and it is equal to  $\lfloor \frac{6(r-1)+2-4}{2} \rfloor+(\lfloor  \frac{6(r-1)+2}{6} \rfloor-1).$ This proves the result for $n=14, 20, 26, \ldots.$

\textbf{Case 2.2} $r \geq 3$  and $k=4.$

By induction hypothesis, the result is true for the class $\mathcal{G}_{6(r-1)+2}$ and we have at least $4r-6$ in $\mathcal{G}_{6(r-1)+2}$ as proved in Case 2.1. Let us take each of these graphs as  $\Delta(G).$ After applying Operation D~\ref{rem2.6}, on each of these, we get  $4r-6$ direct products  $\Delta(G \times H)$ with $6(r-1)+4$ vertices in $\mathcal{G}_{6(r-1)+4}.$ 

As done in the previous cases, let  $\Delta(G)$  be  a character degree graph  with two complete connected components: one having an isolated vertex and the other having $6r-5$ vertices.  After applying Operation D~\ref{rem2.6},  we get a direct product,  $\Delta(G \times H)$ with $6(r-1)+4$ vertices in which $1$ vertex is of  degree $2$ and  the remaining $6(r-1)+3$  vertices are of  degree $6(r-1)+2$. This along with previous construction gives that the number of graphs  in $\mathcal{G}_{6(r-1)+4}$ is at least $4r-5$ and it is equal to  $\lfloor \frac{6(r-1)+4-4}{2} \rfloor+(\lfloor  \frac{6(r-1)+4}{6} \rfloor-1).$ This proves the result for  $n=16, 22,28,\ldots.$

\textbf{Case 2.3} $r \geq 3$  and $k=0.$

By induction hypothesis, the result is true for the class $\mathcal{G}_{6(r-1)+4}$ and we have at least $4r-5$ in $\mathcal{G}_{6(r-1)+4}$ as proved in Case 2.2. Let us take each of these graphs as  $\Delta(G).$ After applying Operation D~\ref{rem2.6}, on each of these, we get  $4r-5$ direct products  $\Delta(G \times H)$ with $6r$ vertices in $\mathcal{G}_{6r}.$ 

As done in the previous cases, let  $\Delta(G)$  be  a character degree graph  with two complete connected components: one having an isolated vertex and the other having $6r-3$ vertices. After applying Operation D~\ref{rem2.6},  we get a direct product,  $\Delta(G \times H)$ with $6r$ vertices in which $1$ vertex is of  degree $2$ and  the remaining $6r-1$  vertices are of  degree $6r-2$. 
This along with previous construction gives that the number of graphs  in $\mathcal{G}_{6r}$ is at least $4r-4.$ 

As proved in Case $1.4,$ we have a character degree graph with $12$ vertices in which $2$ vertices has degree $8$ and $10$ vertices has degree $10.$  Take this $\Delta(G).$  Choose $6$ distinct prime numbers which are not in $\rho(G)$. This gives  character degree graph $\Delta(H)$  as in Figure 5. Now, we have a direct product $\Delta(G \times H)$ with $18$ vertices in which $3$ vertices are of  degree $14$ and the remaining $15$ vertices are of  degree $16.$   Take this as $\Delta(G).$  Again choose $6$ distinct prime numbers which are not in $\rho(G)$ let the character degree graph corresponding to these primes be $\Delta(H)$ (as in Figure 5.) Again, we get a direct product $\Delta(G \times H)$ with $24$ vertices in which $4$ vertices are of  degree $20$ and $20$ vertices are of  degree $22.$ By repeating this process for  $r$ times,  we get a character degree graph with $6r$ vertices in which $r$ vertices are of degree $6r-4$ and $5r$ vertices are of degree  $6r-2.$

Thus the number of graphs  in $\mathcal{G}_{6r}$ is at least $4r-3$  and it is equal to  $\lfloor \frac{6r-4}{2} \rfloor+(\lfloor  \frac{6r+4}{6} \rfloor-1).$ This proves the result for  $n=18, 24,30,\ldots.$
\end{proof}

\section*{Acknowledgment}
This research work is partially supported by DST-FIST (Letter No: SR/FST /MSI- 115 / 2016 dated 10.10.2017) through C. Selvaraj. Further, this research work is supported by CSIR Emeritus Scientist Scheme (No. 21 (1123)/20/EMR-II) of  Council of Scientific and Industrial Research, Government of India through T. Tamizh Chelvam.

\end{document}